\documentclass[11pt,reqno]{amsart}
\usepackage{amsmath,amssymb,rawfonts}
\usepackage{tikz}
\usepackage[english]{babel}
\usepackage{graphicx}
\usepackage[utf8]{inputenc}
\newcommand{\be}{\begin{equation}}
\newcommand{\ee}{\end{equation}}

\newtheorem{theorem}{Theorem}[section]

\newtheorem{defi}[theorem]{Definition}
\newtheorem{prop}[theorem]{Proposition}
\newtheorem{coro}[theorem]{Corollary}

\newtheorem{ex}[theorem]{Example}

\begin{document}  
\title{ TOPOLOGICAL SEMIGROUPS EMBEDDED INTO  TOPOLOGICAL GROUPS}
\author{JULIO C\'ESAR HERN\'ANDEZ ARZUSA\\UNIVERSIDAD DE CARTAGENA\\Programa de Matemáticas}

\maketitle 

\begin{abstract}In this paper we give conditions under which a topological semigroup can be embedded algebraically and topolo\-gically into a compact topological group. We prove that every feebly compact regular first coun\-table cancellative commutative topological semigroup with open shifts is a topological group, as well as every connected locally compact Hausdorff cancellative commutative topological monoid with open shifts. Finally, we use these results to give  sufficient conditions on a topological semigroup that guarantee it to have countable cellularity.
\end{abstract}

\thanks{{\em 2010 Mathematics Subject Classification.} Primary: 54B30, 18B30, 54D10; Secondary: 54H10, 22A30 \\
{\em Key Words and Phrases: Cancellative topological semigroup, cellularity, $\sigma$-compact space, feebly compact space. } }

\section{INTRODUCTION} 

In  \cite{hom} the author presents some properties that allow us to embed, topologically, a cancellative commutative topological semigroup into a topological group, we take advantage of this results to find conditions under which a cancellative commutative topological semigroup has countable cellularity, as well as, when a topological semigroup is a  topological group. The class of topological spaces having countable cellularity is wide, in fact, it contains (among other classes) the class of the $\sigma$-compact paratopolo\-gical groups (see \cite[Corollary 2.3]{T3}), the class of the sequentially compact $\sigma$-compact  cancellative  topolo\-gical monoids  (see \cite[Theorem 4.8]{mio}) and  the class of the subsemigroups of precompact topological groups (see \cite[Corollary 3.6]{T5}). These results have served us as motivation, thus we try to find similar results in the context of topological semigroups.\\ Compactness type conditions (compactness and sequential compactness)  under which a topological semigroup is a topological group are given in \cite[Theorem 2.5.2]{Ar}, \cite[Theorem 2.4]{San} and \cite[Theorem 6]{can}, we use the feeble compactness and local compactness  to obtain similar results. The reflection on the class of the regular spaces allows us to disregard the axioms of separation to obtain topological monoids with countable cellularity.

\section{PRELIMINARIES}We  denote by $\mathbb{Z}$, $\mathbb{R}$ and $\mathbb{N}$, the set of all the integer numbers, real numbers and positive integer numbers, respectively. If $A$ is a set, $|A|$ will denote the cardinal of $A$, $\aleph_0=|\mathbb{N}|$. If $X$ is a topological space and $x\in X$, $N_{x}^{(X)}$ will denote the set of all open neighborhoods of $x$ in $X$ or simply $N_{x}$ when the space is understood. \\
\noindent A \textit{semigroup} is a set $S\neq \emptyset$, endowed with an associative operation. If  $S$ also has neutral element, we say that $S$ is a \textit{monoid}. A mapping $f\colon S\longrightarrow H$ between semigroups is a \textit{homomorphism} if $f(xy)=f(x)f(y)$ for all $x,y\in S$. A \textit{semitopological semigroup} (monoid) consists of a semigroup (resp. monoid) $S$ and a topology $\tau$ on $S$, such that for all $a\in S$, the shifts $x\mapsto ax$ and $x\mapsto xa$ (noted by $l_{a}$ and $r_{a}$, respectively)  are continuous mappings from $S$ to itself. We say that a semitopological semigroup has \textit{open shifts}, if for each $a\in S$ and for each open set  $U$ in $S$, we have that $l_{a}(U)$ and $r_{a}(U)$ are open sets in $S$. A \textit{topological semigroup} (monoid)(\textit{paratopological group}) consists of a semigroup (resp. monoid)(resp. group) $S$ and a topology $\tau$ on $S$, such that the operation of $S$ is jointly continuous. Like \cite{Ar} we do not require that semigroups to be Hausdorff. If $S$ is a paratopological group and if also the mapping $x\mapsto x^{-1}$ is continuous, we say that $S$ is a topological group. A \textit{congruence} on a semigrouop $S$ is an equivalence relation on $S$, $\sim$, such that if $x\sim y$ and $a\sim b$, then $xa\sim yb$. If $S$ is a semitopological semigroup, then we say that $\sim$ is a \textit{closed congruence} if $\sim$ is closed in $S\times S$.  If $\sim$ is an equivalence relation in a semigroup (monoid) $S$ and $\pi\colon S\longrightarrow S/\sim$ is the respective quotient mapping, then $S/\sim$ is a semigroup (monoid) and $\pi$ an homomorphism  if and only if $\sim$ is a congruence (\cite{GONZO}, Theorem 1).\\ The axioms of separation $T_{0}$, $T_1$, $T_{2}$, $T_{3}$ and \textit{regular} are defined in accordance with \cite{En}. We denote by $\mathcal{C}_{i}$ the class of the $T_i$ spaces, where $i\in\{0,1,2,3,r\}$.\\Let $X$ be a topological space, a \textit{cellular family} in $X$ is a pairwise disjoint non empty family of non empty open sets in $X$. The \textit{cellularity} of a space $X$ is noted by $c(X)$ and it is defined by $$c(X)=sup\{|U|: U \mbox{ is cellular familiy in $X$}\}+\aleph_{0}. $$ If $c(X)=\aleph_{0}$, we say that $X$ has \textit{countable cellularity} or $X$ \textit{has the Souslin property}.\\If $X$ is a topological space and $A\subseteq X$, we will note by $Int_{X}(A)$ and $Cl_{X}(A)$, the interior and the closure of $A$ in $X$, or simply $Int(A)$ and $\overline{A}$, respectively,  when the space $X$ is understood. An open set $U$ in $X$, is called \textit{regular open} in $X$ if $Int\overline{U}=U$. It is easy to prove the regular open ones form a base for a topology in $X$, $X$ endowed with this topology,  will note by $X_{sr}$, which we will call  \textit{semiregularitation} of $X.$
 \\From \cite{Epi}, it is well known that for each $i\in \{0,1,2,3, r\}$ and every topological space, $X$, there is a topological space, $\mathcal{C}_{i}(X)\in \mathcal{C}_i$, (unique up to  homeomorphism) and a continuous mapping $\varphi_{(X, \mathcal{C}_{i})}$ of $X$ onto $\mathcal{C}_{i}(X)$, such that given a continuous mapping $f\colon X\longrightarrow Y$, being $Y\in \mathcal{C}_i$, there exists an unique continuous mapping $g\colon \mathcal{C}_{i}(X)\longrightarrow Y	$, such that $g\circ \varphi_{(X, \mathcal{C}_{i})}=f$.\\ According to \cite[Section 2]{comp} a $SAP$-compactification of a semitopological semigroup $S$ is a pair $(G, f)$ consisting of a compact Hausdorff topological group $G$ and a continuous homomorphism $f\colon S\longrightarrow G$ such that for each continuous homomorphism $h\colon S\longrightarrow  K$, being $K$ a compact Hausdorff topological group, there is an unique continuous homomorphism $h^{\ast}\colon G\longrightarrow K$ such that $h=h^{\ast}\circ f$.\\A space $X$ is \textit{feebly  compact} if each locally finite family of open sets in $X$  is finite. By \cite[Theorem 1.1.3]{pseudo}, pseudocompactness is equivalent to feeble compactness in the class of the Tychonoff spaces.\\A space $X$ is \textit{locally compact} if each $x\in X$ has a compact neighborhood.
 \\The following proposition gives us some properties of $\mathcal{C}_i (S)$, when $S$ is a topological monoid with open shifts.

\begin{prop}\label{0206}Let $S$ be a topological monoid with open shifts. Then
\begin{itemize}
\item[i)]$\mathcal{C}_r(S)$ is a monoid,  $\varphi_{(S, \mathcal{C}_i)}$ is an homomorphism (see \cite[Proposition 3.8]{mio2}) and $\mathcal{C}_r(S)=\mathcal{C}_{0}(S_{sr})$  ( see \cite[Theorem 2.8]{mio}). If $S$  is also cancellative, $\mathcal{C}_r(S)$ is cancellative  (see \cite[Lemma 4.6]{mio}).
\item[ii)]If $A$ is an   open set in $S$, then $\mathcal{C}_{r}(A)=\varphi_{(S,\mathcal{C}_r)}(A)$ ( see \cite[Corollary 5.7]{mio2} and \cite[Lemma 4]{semi}).
\item[iii)]$c(S)=c(\mathcal{C}_i (S)$, for each $i\in \{0,1,2,3,r\}$.
\item[iv)]$\mathcal{C}_3(S)=S_{sr}$. Moreover, if $S$ is $T_2$, $\mathcal{C}_r(S)=S_{sr}$ (see \cite[Corollary 2.7]{mio} and \cite[Proposition 1]{semi}).
\item[v)]$\varphi_{(S, \mathcal{C}_i)}$ is open for each $i\in \{0,1,2\}$ (see \cite[Proposition 2.1]{mio}).
\item[vi)] If $S$ is a paratopolgical group, $\mathcal{C}_r(S)$ is a paratopological group (see \cite[Corollary 3.3 and Theorem 3.8 ]{T4}, \cite[Theorem 2.4]{T2}).
\end{itemize}
\end{prop}

\noindent It is easy to see that if  $X$ is a first countable topological space, then $X_{sr}$ is first countable. Since $\mathcal{C}_r(S)=\mathcal{C}_{0}(S_{sr})$ and $\varphi_{(S, \mathcal{C}_0)}$ is open, whenever $S$ is a topological monoid with open shifts, we have the following corollary.

\begin{coro}\label{0207}If $S$ is a first-countable topological monoid with open shifts, then $\mathcal{C}_r(S)$ is first-countable.
\end{coro}

\section{EMBEDDING TOPOLOGICAL SEMIGROUPS INTO TOPOLOGICAL  GROUPS}
Let $S$ be a cancellative  commutative   semigroup,  $S\times S$ is a cancellative commutative   semigroup, by defining the operation coordinate wise. Let us define in $S\times S$ the following relation: $(x,y)\mathcal{R}(a,b)$ if and only if $xb=ya$. It is not to hard to prove that $\mathcal{R}$ is a congruence, hence if $\pi \colon S\times S \longrightarrow (S\times S)/\mathcal{R}$ is the respective quotient mapping, the operation induced by $\pi$ makes of $(S\times S)/\mathcal{R}$ a semigrup. It is easy to prove that $(S\times S)/\mathcal{R}$ is a group, where the equivalence class $\{\pi((x,x)): x\in S\}$ is the neutral element, and the inverse of $\pi((a,b))$ is $\pi((b,a))$. Also, the function $\iota\colon S\longrightarrow (S\times S)/\mathcal{R}$ defined by $\iota (x)=\pi((xa, a))$, for each $x\in S$,  is an algebraic monomorphism, where $a$ is a fixed element of $S$ ($\iota$ does not depend of the choice of $a$). Note that $\pi(x,y)=\pi(xa^{2},ya^{2})=\pi(xa,a)\pi(a,ya)=\iota(x)(\iota(y))^{-1},$ therefore if we identify each element of $S$ with its image under $\iota$, we have that each element of $(S\times S)/\mathcal{R}$ can be written as $xy^{-1}$, where $x,y\in S$, that is to say $(S\times S)/\mathcal{R}=SS^{-1}.$  Let $f\colon S\longrightarrow G$ be a homomorphism to an abelian topological group  $G$, there is an unique homomorphism $f^{\ast}\colon (S\times S)/\mathcal{R}\longrightarrow G$  (defined as $f^{\ast}(\pi((x,y))=xy^{-1}$) such that $f^{\ast}\mid _S=f.$  In summary, we have that the class of the abelian  groups is a reflexive subcategory of the class of the cancellative  commutative semigroups. Since the reflections are unique up to isomorphims, $(S\times S)/\mathcal{R}$ is uniquely  determined by $S$, and we will denote by $\langle S\rangle$ and it is called \textit{ group generated} by $S$. If $S$ is a topological semigroup, we will call $\langle S\rangle^{\ast}$ to  $\langle S\rangle$  endowed with the quotient topology induced by $\pi\colon S\times S \longrightarrow \langle S \rangle$, $S\times S$ endowed with the Tychonoff product topology.  \\

\begin{defi}Let $S$ be a topological semigroup. We say that $S$ has continuous division if give $x,y\in S$ and an open set $V$ in $S$, containing $y$, there are open sets in $S$, $U$ and $W$, containing $x$ and $xy$ ($yx$), respectively; such that $W\subseteq \bigcap\{uV:u\in U\}$ ($W\subseteq \bigcap\{Vu:u\in U\}$).
\end{defi}

\begin{prop} \label{16103}(\cite[Theorem 1.15 and Theorem 1.19]{hom}) If $S$ is a cancellative and commutative Hausdorff  topological semigroup with open shifts, then $\langle S \rangle^{\ast}$ is a Hausdorff topological group and the quotient mapping $\pi \colon S\times S\longrightarrow\langle S\rangle^{\ast}$ is open. 	Furthermore,  if $S$ has continuous division, $\iota\colon S\longrightarrow \iota(S)$ is an homeomorphism and $\iota (S)$ is open in $\langle S \rangle^{\ast}$.
\end{prop}

\begin{prop} \label{16101}Every open subsemigroup of a topological group has continuous division.
\end{prop}

\begin{proof} Let $S$ be an open subsemigroup  of a topological group $G$.
Let $x,y \in S$, since $G$ is a group, $x^{-1}xy=y$. Let  $V$ be an open subset of $S$ containing $y$, then $V$ is open in $G$,  the continuity of the operations of $G$,  implies that there are open subsets in $G$, $K$ and $M$, containing  $x$ and $xy$, respectively, such that $K^{-1}M\subseteq V$. Let us put $U=K\cap S$ and  $W=M\cap S$, then $U$ and $W$ are open subsets of $S$ containing $x$ and $xy$, respectively. Now, if $t\in W$ and $u\in U$, then $u^{-1}t\in K^{-1}M\subseteq V$, hence $t\in uV$, therefore $W\subseteq uV$, for every $u\in U$. We have proved that  $W\subseteq \bigcap_{u\in U} uV$. For $yx$ we proceed analogously.
\end{proof}

\noindent So far we just have embedded, algebraically, semigroups into groups, the following proposition gives us  a topological and algebraic embedding.

\begin{prop}\label{1610} Let $S$ be a cancellative commutative semitopological semigroup with open shifts. There exists a topology $\tau$ in $\langle S\rangle$, such that $(\langle S\rangle, \tau)$ is a semitopological group containing $S$ as an open semigroup. Moreover

\begin{itemize}
\item[i)]$S$ is 1-contable if and only if $(\langle S\rangle,\tau)$ is 1-contable.
\item[ii)] $(\langle S\rangle, \tau)$ is a partopological group if and only if $S$ is a topological semigroup.
\item[iii)]If $S$ is $T_2$, $(\langle S\rangle, \tau)$ is $T_2$.
\item[iv))]If $S$ is Hausdorff and   locally compact, $(\langle S\rangle,\tau)$ is a   locally compact Hausdorff topological group.
\end{itemize}
\end{prop}

\begin{proof}Let $x$ be a fixed element in $S$, being $S$ a cancellative commutative semitopological semigroup with open shifts, and put $\mathcal{B}=\{x^{-1}V:V\in N_x^{(S)}\}$. We will prove that $\{gU: U\in \mathcal{B}, g\in \langle S\rangle\}$ is a base for a topology of semitopological group in $\langle S\rangle$, for it we will prove the conditions 1, 3 an 4 given in \cite[Page 93]{Ravsky}. It is easy to prove the conditions 1 and 4, let us see 3. Let $V_x \in N_x^{(S)}$  and let $t\in x^{-1}V_x$, from the fact that $\langle S\rangle =SS^{-1}$, we have that $t=ab^{-1}$, where $a,b\in S$,  thus $ax\in bV_x$. Given that the shifts in $S$ are open and continuous  we can   find $W_x\in N_x^{(S)}$ such that $aW_x \subseteq  b V_x$, then $t(x^{-1}W_x)\subseteq x^{-1}V_x$ and condition 3 holds. If $\tau$ is the  topology generated by $\mathcal{B}$, then $(\langle S\rangle, \tau)$ is a semitopological group, moreover, since $S\in N_{x}$, $S=x(x^{-1}S)$, we have that $S$ is open in $(\langle S\rangle, \tau)$.\\We will prove that $\{gx^{-1}V: V\in N_x^{(S)}\}$ is a local base at $g$, for every $g\in \langle S\rangle$. Indeed let $U$ be an open set in $(\langle S\rangle, \tau)$ and $g\in U$, then there is $y\in \langle S\rangle$ such that $g\in yx^{-1}V\subseteq U$. There are $a,b, s,t\in S$ such that $y=ab^{-1}$ and $g=ts^{-1}$, therefore $ts^{-1}\in ab^{-1}x^{-1}V_x\subseteq U$, so that $tbx\in asV_x\subseteq bxsU$, hence there is $W_x$ open in  $S$ satisfying  $tbW_x\subseteq asV_x\subseteq bxs U$, therefore  $g(x^{-1}W_x)\subseteq U$, this implies that  $\{gx^{-1}V: V\in N_x ^{S}\}$ is local base at $g$.
\\Let us see that $S$ is a subspace of $(\langle S\rangle, \tau)$. Indeed, let $U$ be an open set in $S$ and let $s\in U$ be, then $xs\in xU$, there is $V\in N_x^{(S)}$ such that $sV\subseteq xU$, this is equivalent to saying that $s(x^{-1}V)\subseteq U$, so that $s\in (s(x^{-1}V))\cap S\subseteq U$, therefore $U$ is open in the topology of subspace of $S$. Reciprocally, let $U$ be an  open set  in $(\langle S\rangle, \tau)$ and $s\in U\cap S$, we can find $W_x \in N_x^{S}$  and $U_s\in N_s^{(S)}$,  such that $s\in sx^{-1}W_x\subseteq U$, $xU_s\subseteq sW_x \subseteq xU$, then $U_s\subseteq U$ and therefore $U_s\subseteq U \cap S$, this proves that $U\cap S$ is open in $S$.  From the fact that $\{gx^{-1}V: V\in N_x^{(S)}\}$ is a local base at $g$, it follows that if $S$ is 1-countable (locally compact), then $(\langle S\rangle, \tau)$ is 1-contable (resp. locally compact) as well. \\Let us suppose that $S$ is a topological semigroup and let us prove that $(\langle S\rangle, \tau)$  is a paratopological group, it can be concluded if we prove that  the condition 2 of \cite[Page 93]{Ravsky} holds for $\mathcal{B}$. Indeed, let $V_x\in N_x^{(S)}$, since $x^{2}\in  xV_x $ and the operation in $S$ is jointly continuous, there exists $W_x\in N_x^{(S)}$ such that $(W_x)^{2}\subseteq xV_x$, then $(x^{-1}W_x)^{2}\subseteq V_x$ and condition 2 holds, this proves that $(\langle S\rangle, \tau)$ is a paratopological group. Let us suppose that $S$ es $T_2$ and let us see $(\langle S\rangle, \tau)$ is $T_2$, indeed let $y,z\in \langle S\rangle$, $z\neq y$, then  there are $a,b,c,d\in S$ such that $z=ab^{-1}$ and $y=cd^{-1}$, so that $ad\neq bc$, by fact that $S$ is $T_2$, we can obtain $V_{ad}\in N_{ad}^{(S)}$ and $V_{bc}\in N_{bc}^{(S)}$. Note that $((bd)^{-1}V_{ad})\cap ((bd)^{-1}V_{bc})=\emptyset$, also,  $(bd)^{-1}V_{ad})\in N_z^{(\langle S\rangle)}$ and   $(bd)^{-1}V_{bc})\in N_y^{(\langle S\rangle)}$, that is to say $(\langle S\rangle, \tau)$  es $T_2$. Finally, if $S$ is locally compact and $T_2$, $(\langle S\rangle, \tau)$ is a semitopological group locally compact and $T_2$, by Elii's Theorem $(\langle S\rangle, \tau)$ is a topological group.\end{proof}

\noindent From the  item $ii)$ of the Proposition \ref{1610} and the Proposition \ref{16101}, we have the following result.

\begin{coro}\label{16102}If $S$ is a cancellative commutative locally compact Hausdorff semitopological semigroup with open shifts, then $S$ has continuous division.
\end{coro}

\begin{coro} Every cancellative commutative locally compact connected Hausdorff topological monoid with open shifts is a topological group.
\end{coro}

\begin{proof}
Let $S$ be a cancellative commutative locally compact connected Hausdorff and  let $\tau$  be the topology given  in the Proposition \ref{1610}. Let $U\in N_{e_S}^{(S)}$, since  $e_S=e_{\langle S\rangle}$ and  $(\langle S\rangle, \tau)$ is a topological group, there is $V\in N_{e_S}^{(\langle S\rangle)}$ satisfying $V^{-1}=V$ and $V\subseteq U$, then $\bigcup_{n\in \mathbb{N}}V^{n}\subseteq \bigcup_{n\in \mathbb{N}}U^{n}$. But $\bigcup_{n\in \mathbb{N}}V^{n}$ is an open subgroup of  $(\langle S\rangle, \tau)$ and therefore is closed in $S$, the connectedness of $S$ implies that $\bigcup_{n\in \mathbb{N}}V^{n}=S$ and $S$ is a topological group.
\end{proof}

\noindent The following theorem tells us that every cancellative commutative locally compact Hausdorff topological semigroup  with open shifts can be embedded as an open semigroup into the locally compact Hausdorff topological group, $\langle S \rangle ^{\ast}$. 

\begin{theorem} \label{1710}Let $S$ be a cancellative, commutative topological semigroup with open shifts. If $S$ is Hausdorff and locally compact, then so is $ \langle S\rangle ^{\ast}$. Moreover $\iota\colon S\longrightarrow \iota(S)$ is an homeomorphism and $\iota(S)$ is open in $\langle S\rangle^{\ast}$.
\end{theorem}

\begin{proof}Since $S$ is locally compact and Hausdorff topological semigroup, so is $S\times S$. By virtue the Proposition \ref{16103}, $\pi\colon S\times S \longrightarrow \langle S\rangle^{\ast}$ is open and $\langle S\rangle ^{\ast}$ is Hausdorff, hence $\langle S\rangle^{\ast}$ is locally compact Hausdorff topological group.  From Corollary \ref{16102} it follows that $S$ has continuous division, therefore the Proposition \ref{16103} guarantees that $\iota\colon S \longrightarrow \iota (S)$ is a homeomorphism and $\iota(S)$ is open in $\langle S\rangle^{\ast}$.
\end{proof}

\noindent It is well known that every pseudocompact  Tychonoff topolo\-gical group can be embedded as a subgroup dense into a compact topological group (see \cite[Theorem 2.3.2]{pseudo}). The following theorem presents an analogue result in cancellative commutative topological semigroups, where also of the pseudocompactness, it is  required the local compactness.

\begin{theorem}\label{23}If $S$ is a cancellative commutative locally compact pseudocompact Hausdorff topological semigroups with open shifts, then $S$ is an open  dense subsemigroup of  $\langle S\rangle^{\ast}$ and $\langle S \rangle^{\ast}$ is a compact topological group.
\end{theorem}

\begin{proof}
Since $S$ is a locally compact pseudocompact space, \cite[Theorem 3.10.26]{En} implies that $S\times S$ is pseudocompact. From the fact that $\langle S\rangle^{\ast}$ is a continuos image of $S\times S$, we have that   $\langle S\rangle^{\ast}$ is pseudompact, therefore the  C$\check{e}$ch-Stone compactification, $\beta \langle S\rangle^{\ast} ,$ is a topological group containing a $\langle S \rangle^{\ast}$ as dense subgroup. \cite[Theorem 3.3.9]{En} guarantees that $\langle S \rangle^{\ast}$ is an open subgroup of $\beta S$, therefore it  is also closed. By the density of $\langle S\rangle^{\ast}$, $\langle S\rangle^{\ast}=\beta
\langle S\rangle^{\ast}$, that it to say, $\langle S\rangle^{\ast}$ is a compact topological group. Since $\langle S\rangle^{\ast}$ is compact, $S$ is open in $\langle S\rangle^{\ast}$ and $\langle S\rangle^{\ast}=SS^{-1}$, there are $s_1 , s_2..., s_n$ in $S$ such that $\langle S  \rangle^{\ast}=\bigcup_{i=1}^{n}Ss_i ^{-1}$. $Cl_{\langle S\rangle^{\ast}}(S)$ is a compact Hausdorff cancellative semigroup, then by \cite[Theorem 2.5.2]{Ar}, $Cl_{\langle S\rangle^{\ast}}(S)$ is a topological group, therefore $s_{i}^{-1}\in Cl_{\langle S\rangle^{\ast}}(S)$ for every $i\in \{1,2,3...,n\}$, hence $(Cl_{\langle S\rangle^{\ast}}(S))s_{i}^{-1}=Cl_{\langle S\rangle^{\ast}}(S)$ for every $i\in \{1,2,3...,n\}$. Since each shift in $\langle S\rangle^{\ast}$ is a homeomorphism, we have that $\langle S\rangle^{\ast}= Cl_{\langle S\rangle^{\ast}}(\bigcup_{i=1}^{n} r_{s_{i}^{-1}}(S) )=\bigcup_{i=1}^{n}Cl_{\langle S\rangle^{\ast}}(r_{s_{i}^{-1}}(S))=\bigcup_{i=1}^{n} r_{s_i ^{-1}}(Cl_{\langle S\rangle^{\ast}}(S))=\bigcup_{i=1}^{n}( Cl_{\langle S\rangle^{\ast}}(S))s_{i}^{-1} =\bigcup_{i=1}^{n} Cl_{\langle S\rangle^{\ast}}(S)= Cl_{\langle S\rangle^{\ast}}(S)$, that it to say, $S$ is dense in $\langle S \rangle^{\ast}$.
\end{proof}

\noindent We obtain the following corollary.

\begin{coro} \label{18021}The closure of any  subsemigorup of a cancellative commutative locally compact pseudocompact Hausdorff topological semigroup with open shifts can be embedded as a dense open subsemigroup into a compact Hausdorff topological group.
\end{coro}

\begin{proof}
Let $S$ be a cancellative commutative locally compact pseudocompact Hausdorff topological semigroup with open shifts and let $K$ be a subsemigorup of $S$. By Theorem \ref{23}, $S$ is an open subsemigroup of $\langle S\rangle^{\ast}$, since $Cl_{S}(K)=Cl_{\langle S\rangle^{\ast}}(K)\cap S$, we have that $Cl_{S}(K)$ is open in $Cl_{\langle S\rangle^{\ast}}(K)$. Now, $K$ is dense in $Cl_{\langle S\rangle^{\ast}}(K)$ and $K\subseteq Cl_{S}(K)\subseteq Cl_{\langle S\rangle^{\ast}}(K)$, this proves that $Cl_{S}(K)$ is dense en $Cl_{\langle S\rangle^{\ast}}(K)$, but  $Cl_{\langle S\rangle^{\ast}}(K)$ is a compact Hausdorff cancellative topological semigroup, so that it is a topological group by \cite[Theorem 2.5.2]{Ar}, and so we have finished the proof.
\end{proof}

\begin{prop}If $S$ is a cancellative commutative locally compact pseudocompact Hausdorff topological semigroup with open shifts, then $(\langle S\rangle^{\ast}, \iota)$ coincides with the $SAP$-compactification of $S$.
\end{prop}

\begin{proof}
Let $S$ be a  cancellative commutative locally compact pseudocompact Hausdorff topological semigroup with open shifts and $f\colon S\longrightarrow G$ a continuous  homomorphism, being $G$ a Hausdorff compact topological  group. Since $S$ is commutative, $f(S)$ is a commutative subsemigroup of $G$, so that $\overline{f(S)}$ is a compact commutative cancellative topological semigroup, which is a topological group by \cite[Theorem 2.5.2]{Ar}. Let us define $f^{\ast}\colon \langle S\rangle^{\ast} \longrightarrow \overline{f(S)}$ by $f^{\ast}(xy^{-1})=f(x)(f(y))^{-1}$, $f^{\ast}$ is a continuous homomorphism and moreover $f^{\ast}\circ \iota =f$, this ends the proof.
\end{proof}

\noindent It is known that each pseudocompact Tychonoff paratopological group is a topological group (see  \cite[Theorem 2.6]{Rez}). The following theorem gives us a similar result  in cancellative commutative topological monoids with open shifts, but instead of group structure we have required the first axiom of countability.

\begin{theorem} \label{18022}Let $S$ be  a  cancellative commutative feebly compact topological monoid with open shifts satisfying the first axiom of countability. Then $\mathcal{C}_{r}(S)$ is a compact metrizable topological group. Moreover, the following statements hold:

\begin{itemize}
\item[i)]If $S$ is $T_{2}$, $S$ is a paratopological group.
\item[ii)]If $S$ is regular, $S$ is a compact metrizable topological group.
\end{itemize}
\end{theorem}

\begin{proof}
Let $S$ be a commutative cancellative topological monoid with open shifts and put $G=\langle S\rangle$. From  Proposition \ref{1610} we have that there is a topo\-logy $\tau$, such that $(G,\tau)$ is a paratopological group containing $S$ as an open monoid. It follows from Proposition \ref{0206} that $\mathcal{C}_r(G)$ is a regular paratopological group containing $\mathcal{C}_r(S)$, So \cite[Corollary 5]{B} implies that $\mathcal{C}_r(G)$ is Tychonoff, therefore so is $\mathcal{C}_r(S)$. Since $\mathcal{C}_{r}(S)$ is feebly compact and Tychonoff, it is pseudocompact. Then $\mathcal{C}_r(S)$  is a pseuducompact subspace of the regular first-countable paratopological group $\mathcal{C}_r(G)$, following 
\cite[Corollary 4.18]{P}, we have that $\mathcal{C}_r(S)$ is metrizable and compact. By Proposition \ref{0206} i), $\mathcal{C}_r(S)$ is cancelative, therefore \cite[Theorem 2.5.2]{Ar} implies that $\mathcal{C}_r(S)$ is a topological group. Now, if $S$ is $T_2$, $\mathcal{C}_r(S)=S_{sr}$, but $S$ and $S_{sr}$ coincide algebraically, thus $S$ is a paratopological group. If $S$ is regular, then $S=\mathcal{C}_r(S)$, this ends the proof.
\end{proof}

\noindent By \cite[Example 2.7.10]{pseudo}, there is a feebly compact Hausdorff 2-countable paratopological group that fails to be a compact topological group, therefore the regularity  in $ii)$ Theorem \ref{18022} cannot be weakened to the Hausdorff separation property.

\begin{ex}Let $\omega_1$ be the first  non countable ordinal, the space $[0, \omega_1)$  of ordinal numbers strictly less than $\omega_1$ with its order topology is regular first-countable  feebly compact space, but $[0, \omega_1)$ is not compact. Then  we can  see the importance of algebraic structure in the Theorem \ref{18022}. 
\end{ex}

\section{CELLULARITY OF  TOPOLOGICAL SEMIGROUPS}
\noindent Finally, we present some results about the  cellularity of topological semigroups.
\begin{theorem} \label{1802} Let $S$ be a cancellative, commutative Hausdorff locally compact $\sigma$-compact topological semigroup with open shifts. Then $S$ has countable cellularity. 
\end{theorem}

\begin{proof}
Since $S\times S$ is locally compact and $\sigma$-compact, from Proposition \ref{16103} we have that $\langle S\rangle^{\ast}$ is locally compact, Hausdorff, $\sigma$-compact topological group, which has countable celluarity following \cite[Corollary 2.3]{T3}. Given that $S$ is open in $\langle S\rangle^{\ast}$, then $c(S)=c(\langle S\rangle^{\ast})=\aleph_0$.
\end{proof}

\noindent In the next corollary we give an analogous result to that of the proposition \ref{1802}, but without considering axioms of separation.

\begin{coro}Every $\sigma$-compact locally compact cancellative commutative topological monoid with open shifts has countable cellularity.
\end{coro}

\begin{proof}
Let $S$ be a  $\sigma$-compact locally compact cancellative commutative topological monoid with open shifts. By Proposition \ref{0206} $\varphi_{(S, \mathcal{C}_2)}$ is open, therefore $\mathcal{C}_2(S)$ is locally compact  Hausdorff topological semigroup with open shifts, this implies that $\mathcal{C}_2(S)$ is regular, so that $\mathcal{C}_2(S)=\mathcal{C}_r(S)$. Since $\mathcal{C}_r(S)$ is cancellative, we can apply the Theorem \ref{1802} and the Proposition \ref{0206}  to conclude that $c(S)=c(\mathcal{C}_r(S))=\aleph_0$.
\end{proof}

\begin{coro}Every subsemigroup of a commutative cancellative locally compact pseudocompact Hausdorff topological semigroup with open shifts has countable cellularity.
\end{coro}

\begin{proof}
Let $S$ be a commutative cancellative locally compact pseudocompact Hausdorff topological semigroup with open shifts and let $K$ be a subsemigroup of $S$. By Corollary \ref{18021}, there exists a compact Hausdorff topolo\-gical group, $G$, containing $Cl_{S}(K)$ as an open semitopological group, therefore $c(K)\leq c(Cl_{S}(K)) \leq c(G)=\aleph_0$. 
\end{proof}

\noindent Since the compact topological groups has countable cellularity and $c(S)=c(\mathcal{C}_r(S))$ for every topological monoid with open shifts, the Theorem \ref{18022} implies the following corollary.

\begin{coro}Let $S$ be  a  cancellative commutative feebly compact topolo\-gical monoid with open shifts satisfying the first axiom of countability. Then $S$ has countable cellularity.
\end{coro}

\end{document}